\documentclass[oneside,10pt,reqno]{amsart}
\usepackage{amssymb,amsmath,amsthm,bbm,enumerate,mdwlist,url,multirow,hyperref,amsthm,paralist,stmaryrd}
\usepackage[pdftex]{graphicx}
\usepackage[shortlabels]{enumitem}
\allowdisplaybreaks

\usepackage{scalerel}
\usepackage{stackengine,wasysym}

\def\ZZ{\mathbb{Z}}

\def\rr{\mathsf{r}}
\def\ee{{\mathsf{e}_\lambda}}
\def\gp{{\mathsf{g}_p}}
\def\EE{\mathcal{E}}
\def\PP{\mathbb{P}}

\def\geom{\mathrm{geom}}

\def\rr{\mathsf{r}}

\def\GG{\mathcal{G}}
\def\MM{\mathcal{M}}

\def\BB{\mathcal{B}}

\def\WW{\mathbb{W}}
\def\FF{\mathcal{F}}

\theoremstyle{definition}
\newtheorem{definition}{Definition}
\theoremstyle{theorem}
\newtheorem{proposition}[definition]{Proposition}

\newtheorem{theorem}[definition]{Theorem}
\newtheorem{corollary}[definition]{Corollary}

\numberwithin{equation}{section}
\numberwithin{definition}{section}
\theoremstyle{remark}
\newtheorem{remark}[definition]{Remark}

\newtheorem{example}[definition]{Example}

\makeatletter
\@namedef{subjclassname@2020}{%
  \textup{2020} Mathematics Subject Classification}
\makeatother
\begin{document}
\title{Symmetric splitting of  one-dimensional noises}

\author{Matija Vidmar}
\address{Department of Mathematics, Faculty of Mathematics and Physics, University of Ljubljana, Slovenia}
\email{matija.vidmar@fmf.uni-lj.si}

\begin{abstract}A symmetric random walk $X$ whose jumps have diffuse law, looked at up to an independent geometric random time,  splits at the minimum into two independent and identically distributed pieces. The same for the maximum. It is natural to ask, are there any other times adapted to $X$ exhibiting this  ``symmetric splitting''? It appears that the phenomenon is most conveniently couched in terms of (what may be called) the  noise structure of $X$. At the level of generality of the latter, an equivalent set-theoretic condition for the symmetric splitting  property is provided, leading to the observation that the answer to the elucidated question is to the affirmative. While we do not deal much with the obvious analog of the phenomenon in continuous time, the discrete findings do beg the  question: does \emph{linear Brownian motion} admit times of symmetric splitting other than the maxima and minima? This is left unresolved, but we do make some comments as to why it may be non-trivial/interesting.
\end{abstract}

\thanks{Financial support from the Slovenian Research and Innovation Agency (ARIS) under programme No. P1-0448 is acknowledged. The author   thanks Jon Warren for many insightful discussions concerning the topic of this paper.}


\date{\today}


\maketitle

\section{Introduction}

\subsection{Motivation, connections to existing literature, scope of results}
It is a classical wisdom, part of the  conglomerate of results falling under the umbrella of the so-called Wiener-Hopf factorization, that a symmetric random walk $X$, with values on the real line and diffuse law of its jumps, splits at the time $\tau^0_\mathsf{g}$ of the minimum  (or maximum) before an independent geometric random time $\mathsf{g}$ into two independent and identically distributed pieces: the increments of $X$ before $\tau^0_\mathsf{g}$ (looked backwards) together with $ \tau^0_\mathsf{g}$, and the increments of $X$ after $\tau^0_\mathsf{g}$ together with $\mathsf{g}- \tau^0_\mathsf{g}$. A similar result holds true in the case of L\'evy processes if we make suitable qualifications to preclude a jump at the a.s. unique supremum/infimum and if we substitute for $\mathsf{g}$ an independent exponential random time $\mathsf{e}$ \cite[Lemma~VI.6]{bertoin} \cite{greenwood-pitman}. 

The purpose of  this short note is to investigate the natural generalization of the phenomenon just delineated. Namely, given a symmetric real-valued process with stationary independent increments (PSII) $X$, we will be interested in the structure of times $\tau$ having the above-enunciated property when substituted for $\tau^0$, to be referred to henceforth as  ``symmetric splitting'' for $X$ by $\tau$. In particular, we want to know whether this property is special to the maxima and minima of $X$ or not. We restrict attention to symmetric  $X$, i.e. those for which $X$ and $-X$ have the same law, because by duality it is only for them that time-reversal is measure-preserving \cite[Lemma~3.4]{kyprianou}.

As it happens, in discrete time we shall be able to provide a set-theoretic (up to a.s. equality, of course) condition for a family of adapted times $\tau=(\tau_n)_{n\in \mathbb{N}_0}$ to be symmetric splitting for $X$. Here, for $n\in \mathbb{N}_0$, the reader should think of $\tau_n$ as being the generalization of the time $\tau_n ^0$ when $X$ attains its minimum on $[0,n]$. By adaptedness we intend simply that $\tau_n$ takes its values in $[0,n]\cap \mathbb{Z}$ and is recoverable from the path of $X$ up to time $n$. It will be found that $\tau$ is symmetric splitting if and only if, in a sense that we shall make precise, $\tau_n$ selects that unique time point $k$ in $[0,n]\cap \mathbb{Z}$ for which the increments of $X$ looked backwards from $k$ and those looked forwards from $k$ ``observe the same property'' up to $k$ and $n-k$ respectively. This is the content of Theorem~\ref{thm:discrete} of Subsection~\ref{subsection:the-characterization}. Once this has been achieved, it is then not difficult to put in evidence many examples of symmetric splitting times for $X$, see Example~\ref{example:basic-other-families}. With the benefit of hindsight they will appear to the reader as nearly obvious. 

Actually, we shall cast our results somewhat more generally. Indeed, it seems that the whole matter is most conveniently rendered and manipulated in terms of a certain discrete  one-dimensional time-symmetric ``noise'' (in the sense of Tsirelson \cite{picard2004lectures,tsirelson-nonclassical}) that is canonically attached to the random walk $X$, and it is only this abstract noise structure, formally specified in Subsection~\ref{subsection:setting}, that is relevant to the problem at hand. 

In continuous time our ambitions shall be much more reserved and we shall content ourselves with formulating, in Section~\ref{section:bm}, the problem for linear Brownian motion, which appears to be particulary interesting since the discrete results do not appear to offer any insight into providing examples of times other than the maxima and minima that would split it symmetrically. We shall also make some comments as to why proving the converse, that there are none, might equally be quite subtle (if true).

Dropping the symmetry requirement and asking only for independence we arrive at the notion of what we may call simply a ``splitting time''. These are plentiful even for a 1d Brownian motion $B$, one family of examples being the times of the minima of $B$ to which is added a deterministic drift, some others follow from the findings of  \cite[esp. Theorem~5.7]{vidmar-warren}. More broadly, random times at which a PSII (resp. Markov process) splits into an independent past and future (resp. conditionally on the present) have received  a considerable amount of interest, see e.g. \cite[Section~III.49]{rogers-williams} \cite {splitting-non-standard,greenwood-shaked,KENNEDY,independence-vidmar,perman} and the references therein (observe that in this context $\mathsf{g}$/$\mathsf{e}$ can be thought of as just the first arrival time of an independent Bernoulli/homogeneous Poisson process). Notably, \cite{greenwood-shaked,KENNEDY}  investigate so-called dual  stopping times: a kind of self-duality corresponds then to a subclass of symmetric splitting families that we call  honest and explore in Subsection~\ref{honest}. See also \cite{kallenberg-splitting} for the setting of regenerative sets. 

It is the  requirement of equality of distribution, over and above independence, which gives symmetric splitting times, considered herein, added appeal.

\subsection{General notation} We write $\PP[f]$ for the expectation of a measurable numerical $f$ under a probability $\PP$. For a $\sigma$-algebra $\GG$, $\GG_+$ (resp. $\GG_b$) is the family of all  $\GG$-measurable $[0,\infty]$-valued (resp. real bounded) maps. $2^A$ denotes the power set of $A$, as usual. Given  a sequence $\theta\in \mathbb{R}^\mathbb{Z}$ we interpret $\sum_{k=1}^0\theta(k):=0$ and $\sum_{k=1}^{n}\theta(k):=-\sum_{k=n+1}^0\theta(k)$  for $n\in -\mathbb{N}$, which ensures that $\sum_{k=1}^n\theta(k)-\sum_{k=1}^m\theta(k)=\sum_{k=m+1}^n\theta(k)$ for all $m< n$ from $\ZZ$, while $\sum_{k=1}^{-n}\theta(k)=-\sum_{k=1}^n\theta(-k+1)$ for all $n\in \ZZ$. In the context of random times we adhere to the usual convention $\inf\emptyset:=\infty$.

\section{Splitting symmetrically a  discrete noise}\label{discrete}

\subsection{Setting}\label{subsection:setting}
As mentioned, we shall formulate everything in the milieu of the discrete notion of a one-dimensional noise, \`a la Tsirelson, equipped with a suitable symmetry of time-reversal. For immediate motivation of this abstract structure the reader should consult Example~\ref{example:basic} just below.

On a probability space $(\Omega,\GG,\PP)$ let then there be given a family  $\FF=(\FF_{m,n})_{(m,n)\in \mathbb{Z}^2,m\leq n}$ of $\PP$-complete (i.e., each containing $\PP^{-1}(\{0,1\})$) sub-$\sigma$-fields of $\GG$ and  a $\ZZ$-indexed  group $\Delta=(\Delta_h)_{h\in \mathbb{Z}}$ of  transformations  on $\Omega$, each measurable-$\GG/\GG$, preserving the measure $\PP$, and satisfying between them the following properties:
\begin{enumerate}
\item\label{properties:i} $\FF_{-\infty,\infty}:=\lor_{n\in \mathbb{N}_0}\FF_{-n,n}=\GG$ (which may, ceteris paribus, always be ensured by restriction of $\PP$);
 \item $\Delta_h^{-1}\FF_{m,n}=\FF_{m+h,n+h}$ for all integer $m\leq n$ and  $h$  (homogeneity); 
 \item $\FF_{l,m}$ and $\FF_{m,n}$ are independent and their join is $\FF_{l,n}$, this for all integer $l\leq m\leq n$ (factorization).
\suspend{enumerate} 
 We would say that $\FF$ constitutes a one-dimensional noise (in discrete time) under $\PP$, $\Delta$ being its time-shifts.  For $n\in \ZZ$ write $\FF_{n,\infty}:=\lor_{m\in \ZZ,m\geq n}\FF_{n,m}$ and $\FF_{-\infty,n}:=\lor_{m\in \ZZ,m\leq n}\FF_{m,n}$; also  
  $\FF_n:=\FF_{0,n}$ for  $n\in \mathbb{N}_0\cup \{\infty\}$. Of course, $\FF_{0}=\PP^{-1}(\{0,1\})$, the $\PP$-trivial $\sigma$-field.
 
Let there be given also an involution $\rr$ on $\Omega$ (so, $\rr\circ \rr=\mathrm{id}_\Omega$), measurable-$\GG$/$\GG$, preserving the measure $\PP$ and such that 
\resume{enumerate}
\item $\rr\circ \Delta_h=\Delta_{-h}\circ \rr$ for all integer $h$ (intertwining relation);
\item\label{properties:v} $\rr^{-1}\FF_{m,n}=\FF_{-n,-m}$ for all integer  $m \leq n$ (time-reflection property).
\end{enumerate}
We  summarize this by saying that the noise $(\PP;\FF,\Delta)$ is symmetric under the time-reversal $\rr$. Notice that as a consequence of the intertwining relation and the group character of $\Delta$, like $\rr$, so too for each $m\in \mathbb{Z}$, the ``reflection about $m$'', namely the measure-preserving map $\rr\circ \Delta_m$, is  involutive.

The most significant instantiation of the above  structure to keep in mind is that of 
\begin{example}\label{example:basic}
Let $X=(X_n)_{n\in \mathbb{Z}}$ be the canonical process on $\Omega':=\{\omega\in \mathbb{R}^\mathbb{Z}:\omega(0)=0\}$, endowed with the trace $\GG'$ of the tensor $\sigma$-field $\BB_\mathbb{R}^{\otimes \mathbb{Z}}$, and let $\PP'$ be a probability on $(\Omega',\GG')$ making $X$ into a  two-sided symmetric random walk issuing from $0$. For $m\leq n$ from $\ZZ$ let $\FF_{m,n}'$ be the $\sigma$-field generated by the increments of $X$ on the interval $[m,n]$ and by the trivial events ${\PP'}^{-1}(\{0,1\})$, for $h\in \mathbb{Z}$ let $\Delta_h':=(\Omega\ni \omega\mapsto \omega(\cdot+h)-\omega(h))$  be the L\'evy shift by $h$, finally set $\rr':=(\Omega\ni \omega\mapsto (\mathbb{Z}\ni n\mapsto \omega(-n)))$ for time-reversal. Then $(\PP';\FF',\Delta')$ is a noise symmetric under $\rr'$.
\end{example}
It is worthwhile to reinterpret Example~\ref{example:basic}  in terms of the increments of the random walk, rather than its paths. It distills to 

\begin{example}\label{ex.increments}
Let $(E,\EE,\nu)$ be a probability space, $\xi=(\xi_n)_{n\in \mathbb{Z}}$ the canonical process on $\Omega':=E^\ZZ$ endowed with the tensor $\sigma$-field $\GG':=\EE^{\otimes \ZZ}$ and the product probability $\PP':=\nu^{\times \mathbb{Z}}$. Thus $\xi$ is nothing  but a two-sided sequence of i.i.d. random elements having the law $\nu$ under $\PP'$. For $m\leq n$ from $\mathbb{Z}$ let $\FF_{m,n}'$ be the $\PP'$-complete $\sigma$-field generated by $\xi_{m+1},\ldots,\xi_n$ and set $\Delta_h':=(\Omega\ni \theta\mapsto \theta(\cdot+h))$ for $h\in \ZZ$.  Finally let $\iota$ be an involution on $E$, measurable-$\EE$/$\EE$, preserving the measure $\nu$ and put $\rr':=(\Omega\ni \theta\mapsto (k\ni \mathbb{Z}\mapsto \iota(\theta(-k+1))))$. Then $(\PP';\FF',\Delta')$ is a noise symmetric under $\rr'$. (Example~\ref{example:basic} is recovered by taking for $\nu$ the law of $X_1$ therein and  $\iota=-\mathrm{id}_{\mathbb{R}}$, up to the mapping $(\mathbb{R}^\mathbb{Z}\ni \theta\mapsto (\ZZ\ni n\mapsto \sum_{k=1}^n\theta(k)))$ that ``preserves everything in sight''.)
\end{example}

\begin{remark}\label{rmk:level-of-generality}
If $\PP$ is standard, then we can form the quotient $(E,\EE,\nu):=(\Omega,\FF,\PP)/_{\FF_1}$. The associated quotient map $[\cdot]:\Omega\to E$ sends the $\FF_1$/$\FF_1$-measurable map $\rr\circ \Delta_1$ to the involution $\iota$ of $E$ that preserves the measure $\nu$: $\iota([\omega])=[(\rr\circ \Delta_1)(\omega)]$ for $\PP$-a.e. $\omega$.  Since $\GG$ is the independent join of $\lor_{n\in \ZZ}\FF_{n,n+1}$ we get the natural mod-$0$ isomorphism $\psi:=(\Omega\ni \omega\mapsto (\ZZ\ni k\mapsto [\Delta_{k-1}(\omega)]))$ between $\PP$ and $\PP':=\nu^{\times \mathbb{Z}}$, which carries $\FF$, $\Delta$ and $\rr$ respectively onto the $\FF'$, $\Delta'$ and $\rr'$ of Example~\ref{ex.increments} associated to this $\nu$ and  $\iota$ (this is what we meant also there under ``preservation of everything in sight''). 
\end{remark}
We  further subject to our analysis a sequence $\tau=(\tau_n)_{n\in \mathbb{N}_0}$ of random times, $\tau_n$ taking values in $[0,n]\cap \mathbb{Z}$ and being $\FF_{n}$-measurable for all $n\in \mathbb{N}_0$. Notice that $\tau_0=0$ identically. 

Lastly, for each fixed $p\in (0,1)$, under an enlargement of $\PP$ --- that, by a standard abuse of notation, we shall not make explicit in the notation --- we grant ourselves access  to an independent geometric random time $\gp$ with  values in $\mathbb{N}_0$, success parameter $p$. To avoid any ambiguity, the geometric law designation in the preceding stipulates that $$\PP(\gp=k)=p(1-p)^k,\quad k\in \mathbb{N}_0,$$ and we  write $\gp\sim \geom_{\mathbb{N}_0}(p)$ for short.
\subsection{The characterization}\label{subsection:the-characterization}
 
 We are ready to formally introduce the fundamental concept that we wish to concern ourselves with. 
\begin{definition}\label{def-}
We say that $\tau$ is a symmetric splitting family for $(\PP;\FF,\Delta,\rr)$ if, for all $p\in (0,1)$, relative to the $\sigma$-field $\FF_\infty\otimes 2^{\mathbb{N}_0}$, the pair $(\rr\circ \Delta_{\tau_\gp},\tau_\gp)$  is independent of and has the same distribution as the pair $(\Delta_{\tau_\gp},\gp-\tau_\gp)$.
\end{definition}
Let us stress/reiterate that in Definition~\ref{def-} the two pairs under inspection are viewed as random elements in $(\Omega\times (0,\infty),\FF_\infty\otimes 2^{\mathbb{N}_0})$. Since this property of symmetric splitting is left invariant by mod-$0$ isomorphisms ``preserving everything in sight'' Remark~\ref{rmk:level-of-generality} entails that actually  Example~\ref{ex.increments} is  quite general in view of our mandate.

Now, if $\tau$ is a symmetric splitting family, then, for each $p\in (0,1)$,
 the independence and equality in distribution of $\tau_\gp$ and $\gp-\tau_\gp$  ensure that, for all $v\in (0,1]$, $$\PP[v^{\tau_\gp}]^2=\PP[v^{\tau_\gp}]\PP[ v^{\gp-\tau_\gp}]=\PP[v^{\tau_\gp} v^{\gp-\tau_\gp}]=\PP[v^\gp],$$ which is uniquely determined (a universal constant depending only on $p\in (0,1)$ and $v\in (0,1]$). So, the distribution of $\tau_\gp$ is uniquely determined (depends only on $p\in (0,1)$). In turn, for each $n\in \mathbb{N}_0$, the distribution of $\tau_n$ is uniquely determined (a universal constant). It is nothing but the distribution of the time of the minimum on $[0,n]$ of any symmetric random walk whose times of minima are a.s. unique; in particular, $\PP(\tau_n=k)>0$ for all $k\in[0,n]\cap \ZZ$.  Indeed:

 \begin{example}\label{example:minima}
In the context of Example~\ref{example:basic}, assuming the increments of $X$ have diffuse law, we may set, for each $n\in \mathbb{N}_0$, $\tau^0_n$ equal to the time of the minimum of $X$ on $[0,n]$ (when unique, and, say, for definiteness, $\tau^0_n:=0$ on the exceptional negligible set on which it fails). Then $\tau^0=(\tau^0_n)_{n\in \mathbb{N}_0}$ is a symmetric splitting family for $X$ (that is to say, for the symmetric noise attached to it). 
\end{example}
Notice that in the preceding $\tau_n^0$ is a.s. got by selecting that unique time point on $[0,n]$ out of which the increments of $X$ looked  backwards and forwards observe the same property (to wit, stay  above zero) up to times $k$ and $n-k$ respectively. This may be taken as representing the motivation and intuition for the finding of

\begin{theorem}\label{thm:discrete}
The following two statements are equivalent.
\begin{enumerate}[(a)]
 \item\label{thm:discrete:a} The family $\tau$ is symmetric splitting for $(\PP;\FF,\Delta,\rr)$.
 \item\label{thm:discrete:b} For $\{m,n\}\subset \mathbb{N}_0$, $\PP$-a.s., 
 \begin{equation}\label{main}
 \{\tau_{m+n}=m\}= (\rr\circ \Delta_m)^{-1}(\{\tau_m=0\})\cap \Delta_m^{-1}(\{\tau_{n}=0\}).
 \end{equation}
\end{enumerate}
In particular, if $\tau$ is symmetric splitting for $(\PP;\FF,\Delta,\rr)$, then automatically 
\begin{equation}\label{symmetric-split-eq}
\tau_n+\tau_n( \rr\circ \Delta_n)=n\text{ a.s.-$\PP$ for all $n\in \mathbb{N}_0$.}
\end{equation}
\end{theorem}
\begin{proof}
\ref{thm:discrete:b} $ \Rightarrow$ \ref{thm:discrete:a}. Let $p\in (0,1)$,  $\{f,g\}\subset (\FF_\infty)_+$ and $\{m,n\}\subset \mathbb{N}_0$. We compute 
\begin{align}\nonumber
& \PP [f(\rr\circ \Delta_{\tau_\gp})g(\Delta_{\tau_\gp});\tau_\gp=m,\gp-\tau_\gp=n]\\\nonumber
&= \PP [f(\rr\circ \Delta_m)g(\Delta_m);\tau_{m+n}=m,\gp=m+n]\\\nonumber
 &=p(1-p)^{m+n} \PP [f(\rr\circ \Delta_m)g(\Delta_m);\tau_{m+n}=m]\quad (\because \text{ $\gp\sim \geom_{\mathbb{N}_0}(p)$ is independent  of $\GG$})\\\nonumber
  &=p(1-p)^{m+n} \PP [f(\rr\circ \Delta_m)g(\Delta_m); (\rr\circ \Delta_m)^{-1}(\{\tau_m=0\})\cap \Delta_m^{-1}(\{\tau_n=0\})]\quad (\because\, \eqref{main})\\\nonumber
 &=p(1-p)^{m+n} \PP [f(\rr\circ \Delta_m);(\rr\circ \Delta_m)^{-1}(\{\tau_m=0\})]\PP[g(\Delta_m);\Delta_m^{-1}(\{\tau_n=0\})]\\
&=p(1-p)^{m+n} \PP [f;\tau_m=0]\PP[g;\tau_n=0]\quad (\because \text{$\rr$ and $ \Delta_m$ are measure-preserving}),\label{eq:ancillary}
\end{align}
where in the penultimate equality we take into account that by homogeneity and the time-reflection property  $f(\rr\circ \Delta_m) \in (\FF_{-\infty,m})_+$, $(\rr\circ \Delta_m)^{-1}(\{\tau_m=0\})\in  \FF_{m}\subset \FF_{-\infty,m}$, while $g(\Delta_m)\in \FF_{m,\infty}$, $\Delta_m^{-1}(\{\tau_n=0\})\in \FF_{m,m+n}\subset \FF_{m,\infty}$, together with the factorization property of $\FF$ which ensures the independence of $\FF_{-\infty,m}$ and $\FF_{m,\infty}$.

Taking in \eqref{eq:ancillary}, consecutively: $f\equiv 1$ and $\sum_{m\in \mathbb{N}_0}$;  $g\equiv 1$ and $\sum_{n\in \mathbb{N}_0}$; both of the preceding --- then recombining, we deduce that
$$\PP [f(\rr\circ \Delta_{\tau_\gp});\tau_\gp=m]=\PP [f(\Delta_{\tau_\gp});\gp-\tau_\gp=m]$$ and
$$\PP [f(\rr\circ \Delta_{\tau_\gp})g(\Delta_{\tau_\gp});\tau_\gp=m,\gp-\tau_\gp=n]=\PP [f(\rr\circ \Delta_{\tau_\gp});\tau_\gp=m]\PP [g(\Delta_{\tau_\gp});\gp-\tau_\gp=n].$$
By a monotone class argument this is enough to establish that, relative to the $\sigma$-field $\FF_\infty\otimes 2^{\mathbb{N}_0}$, $(\rr\circ \Delta_{\tau_\gp},\tau_\gp)$  is independent of and has the same distribution as $(\Delta_{\tau_\gp},\gp-\tau_\gp)$, which is \ref{thm:discrete:a}.

\ref{thm:discrete:a} $\Rightarrow$ \ref{thm:discrete:b}. We reverse course.

On the one hand, for all $p\in (0,1)$,  $\{m,n\}\subset \mathbb{N}_0$ and $\{g,f\}\subset (\FF_\infty)_b$ we have, from the independence of $(\rr\circ \Delta_{\tau_\gp},\tau_\gp)$ and $(\Delta_{\tau_\gp},\gp-\tau_\gp)$ relative to  $\FF_\infty\otimes 2^{\mathbb{N}_0}$, that
\begin{equation*}
\PP [f(\rr\circ \Delta_{\tau_\gp})g(\Delta_{\tau_\gp});\tau_\gp=m,\gp-\tau_\gp=n]=\PP [f(\rr\circ \Delta_{\tau_\gp});\tau_\gp=m]\PP [g(\Delta_{\tau_\gp});\gp-\tau_\gp=n],
 \end{equation*}
 which is to say
\begin{align*}
&p(1-p)^{m+n}\PP [f(\rr\circ \Delta_m)g(\Delta_{m});\tau_{m+n}=m]\\
&=\left(\sum_{k=m}^\infty p(1-p)^{k}  \PP[f(\rr\circ \Delta_m);\tau_k=m]\right)\left(\sum_{k=n}^\infty p(1-p)^{k}  \PP [g(\Delta_{k-n});\tau_{k}=k-n]\right).
\end{align*} 
Dividing both sides by $(1-p)^{m+n}$ and letting $p\uparrow 1$ we infer 
\begin{equation}\label{fundamental:converse}
\PP [f(\rr\circ \Delta_m)g(\Delta_{m});\tau_{m+n}=m]= \PP[f(\rr\circ \Delta_m);\tau_m=m] \PP [g;\tau_n=0].
\end{equation}
  Taking in \eqref{fundamental:converse} in turn $f\equiv 1$, then $g\equiv 1$, then  both of these; and recombining, we see that 
$$\PP(\tau_{m+n}=m)\PP [f(\rr\circ \Delta_m)g(\Delta_{m});\tau_{m+n}=m]=\PP [f(\rr\circ \Delta_m);\tau_{m+n}=m]\PP [g(\Delta_{m});\tau_{m+n}=m].$$
Using $\{m+n\}\in \FF_{m+n}$ (recall $\tau$ is adapted) and the fact that by the factorization property of $\FF$, $\FF_{m+n}$ is the independent join of $\FF_{m}$ and $\FF_{m,m+n}$ the latter in turn renders (via a monotone class argument) 
  \begin{equation*}
\PP(\tau_{m+n}=m)  \mathbbm{1}_{\{\tau_{m+n}=m\}}=\PP(\tau_{m+n}=m\vert \FF_{m})\PP(\tau_{m+n}=m\vert \FF_{m,m+n})\text{ a.s.-$\PP$},
  \end{equation*}
  which can only be if there are strictly positive real numbers $a$ and $b$ and events $A\in \FF_m$ and $B\in \FF_{m,m+n}$,  such that  $\PP(\tau_{m+n}=m\vert \FF_{m})=a\mathbbm{1}_{A}$  and 
  $\PP(\tau_{m+n}=m\vert\FF_{m,m+n})=b\mathbbm{1}_B$ a.s.-$\PP$. We deduce that 
  \begin{equation}\label{eq:identify-independence}
\{\tau_{m+n}=m\}=  (\rr \circ \Delta_m)^{-1}(A_{m,n})\cap \Delta_m^{-1}(B_{m,n})\text{ a.s.-$\PP$}
  \end{equation}
  for some a.s.-$\PP$ unique $A_{m,n}\in \FF_m$, $B_{m,n}\in \FF_n$.  
  
  On the other hand, for again $p\in (0,1)$,  $m\in \mathbb{N}_0$ and $f\in (\FF_\infty)_b$,  from the equality in distribution of $(\rr\circ \Delta_{\tau_\gp},\tau_\gp)$ and $(\Delta_{\tau_\gp},\gp-\tau_\gp)$ relative to  $\FF_\infty\otimes 2^{\mathbb{N}_0}$, 
  $$\PP [f(\rr\circ \Delta_{\tau_\gp});\tau_\gp=m]=\PP [f(\Delta_{\tau_\gp});\gp-\tau_\gp=m],$$
  i.e. $$\sum_{k=m}^\infty p(1-p)^{k}  \PP[f(\rr\circ \Delta_m);\tau_k=m]=\sum_{k=m}^\infty p(1-p)^{k}  \PP [f(\Delta_{k-m});\tau_{k}=k-m].$$ This being  so for all $p\in (0,1)$, comparing the coefficients in these two $(1-p)$-power series, we infer that, for all $n\in \mathbb{N}_0$,
  $$ \PP[f(\rr\circ \Delta_m);\tau_{m+n}=m]= \PP [f(\Delta_{n});\tau_{m+n}=n],$$
  viz., on inserting \eqref{eq:identify-independence}, taking into account the factorization property of $\FF$ and since the time-shifts and the reflection of time are measure-preserving, 
 $$ \PP[f;A_{m,n}]\PP(B_{m,n})= \PP [f;B_{n,m}]\PP(A_{n,m}).$$  
Because  $f\in( \FF_\infty)_b$ was arbitrary, we are forced to conclude that $$B_{n,m}=A_{m,n}\text{ a.s.-$\PP$}.$$ Relabeling the $A$ sets we conclude that 
  \begin{equation}\label{eq:identify-independence+sym}
\{\tau_{m+n}=m\}=  (\rr \circ \Delta_m)^{-1}(A_{m+n,m})\cap \Delta_m^{-1}(A_{m+n,n})\text{ a.s.-$\PP$}
  \end{equation}
    for some $A_{k,l}\in \FF_l$ as $k\geq l$ run over $\mathbb{N}_0$.  Here we may and do insist that $A_{k,0}=\Omega$ for all $k\in \mathbb{N}_0$. 
  
Return once again to \eqref{fundamental:converse}, which in conjunction with \eqref{eq:identify-independence+sym} tells us that 
 \begin{equation*}
 \PP [f;A_{m+n,m}]\PP[g;A_{m+n,n}]= \PP[f;A_{m,m}] \PP [g;A_{n,n}].
  \end{equation*}
  But this  can only be if $A_{m+n,m}=A_{m,m}=:A_m$ a.s.-$\PP$. In other words, we can recast \eqref{eq:identify-independence+sym} as   \begin{equation*}
\{\tau_{m+n}=m\}=  (\rr \circ \Delta_m)^{-1}(A_m)\cap \Delta_m^{-1}(A_n)\text{ a.s.-$\PP$.}
  \end{equation*}
 Plugging in $m=0$ we identify $A_n=\{\tau_n=0\}$ a.s.-$\PP$ and the proof of the necessity of \ref{thm:discrete:b}  for \ref{thm:discrete:a} is complete.
 
 The final assertion of the theorem follows at once from the characterization \ref{thm:discrete:b}  coupled with the intertwining relation and group properties of $\rr$ and $\Delta$. Explicitly, we check that $$(\rr\circ \Delta_{m+n})^{-1}(\{\tau_{m+n}=n\})=\{\tau_{m+n}=m\} \text{ a.s.-$\PP$ for } \{m,n\}\subset \mathbb{N}_0,$$ which entails \eqref{symmetric-split-eq}.
\end{proof}

\begin{corollary}\label{corollary:preclude}
The existence of a symmetric splitting family $\tau$ for $( \PP;\FF,\Delta,\rr)$ precludes the possibility of $\PP\vert_{\FF_1}$ having atoms. 
\end{corollary}
As a trivial case we get that the simple symmetric random walk on the integers admits no symmetric splitting.
\begin{proof}
For suppose, per absurdum, $B\in \FF_1$ is an atom of $\PP\vert_{\FF_1}$ and $\tau$ is a symmetric splitting family. Then $A:=B\cap (\rr\circ \Delta_2)^{-1}(B)$ is an atom of $\PP\vert_{\FF_2}$. Hence, measurability of $\tau_2$ for $\FF_2$ entails that $\tau_2$ is a.s.-$\PP$ constant on the event $A$ that satisfies $(\rr\circ \Delta_2)^{-1}(A)=A$.  \eqref{symmetric-split-eq} then requires the value of this constant to be  $1$. In turn, by \eqref{main}, we are forced to conclude that $\tau_1=0$ a.s.-$\PP$ on $ (\rr\circ \Delta_1)^{-1}(B)$. By the same token as for $A$, a.s.-$\PP$ on the atom $A':=(\rr\circ \Delta_1)^{-1}(B)\cap \Delta_1^{-1}(B)$ of $\PP\vert_{\FF_2}$ that also satisfies $(\rr\circ \Delta_2)^{-1}(A')=A'$ we have $\tau_2=1$, which in turn requires $\tau_1=0$ a.s.-$\PP$ on $B$. Altogether we must have $\tau_1+\tau_1(\rr\circ \Delta_1)=0+0=0$ a.s.-$\PP$ on  $B$ in contradiction with \eqref{symmetric-split-eq}. 
\end{proof}

Example~\ref{example:minima} and Theorem~\ref{thm:discrete} lead to a family of examples of symmetric splitting times got from random walks that are adapted, in a sense to be made precise at once, to $(\PP;\FF,\Delta,\rr)$.
\begin{definition}
A two-sided random walk $X=(X_n)_{n\in \mathbb{Z}}$ with values  in the real line is said to be adapted to $(\PP;\FF,\Delta,\rr)$ if $X\circ \Delta_h=X_{h+\cdot}-X_h$ a.s.-$\PP$ for all $h\in \mathbb{Z}$, the increments of $X$ on $[m,n]$ are measurable-$\FF_{m,n}$ for all integer $m\leq n$, finally $X\circ \rr=X_{-\cdot}$ a.s.-$\PP$. 
\end{definition}
Less precisely, but more succinctly: the increments of an adapted $X$ are measurable relative to $\FF$, but also $X$ commutes with the time-shifts and the reflection of time. 
\begin{remark}\label{rmk:adapted-rws}
For an adapted $X$ automatically $X_0=0$ a.s.-$\PP$ and $\xi:=X_1$ is an $\FF_1$-measurable random variable, symmetric with respect to $(\PP;\FF,\Delta,\rr)$ in the sense that $\xi=-\xi(\rr\circ \Delta_1)$ a.s.-$\PP$, and satisfying $X=(\sum_{k=1}^n \xi( \Delta_{k-1}))_{n\in \ZZ}$ a.s.-$\PP$. Conversely, if an $\FF_1$-measurable random variable $\xi$ is symmetric with respect to $(\PP;\FF,\Delta,\rr)$, then $(\sum_{k=1}^n \xi(\Delta_{k-1}))_{n\in \ZZ}$ is adapted to $(\PP;\FF,\Delta,\rr)$. 
\end{remark}
\begin{corollary}\label{corollary}
Suppose the random walk $X$ is adapted to $(\PP;\FF,\Delta,\rr)$ and that its jumps have diffuse law. Setting, for $n\in \mathbb{N}_0$, $\tau_n^0$ equal to the a.s. unique time of the minimum of $X$ on $[0,n]$ (and $:=0$ on the exceptional set on which uniqueness fails, for instance), then the family $(\tau_n^0)_ {n\in \mathbb{N}_0}$ is symmetric splitting for $(\PP;\FF,\Delta,\rr)$.  
\end{corollary}
\begin{proof}
One has only to apply Theorem~\ref{thm:discrete} exploiting the fact that $X$ commutes with time-reflection and the time-shifts.
\end{proof}
In particular: 
\begin{example}\label{example:basic-other-families}
Returning to Example~\ref{example:basic} we see that effecting on all the increments of $X$ simultaneously any fixed odd measurable transformation $\kappa:\mathbb{R}\to\mathbb{R}$ and recombining back these increments yields the adapted symmetric random walk $\kappa\cdot X:=(\sum_{k=1}^n\kappa(X_k-X_{k-1}))_{n\in\mathbb{Z}}$. For $\kappa=\mathrm{id}_\mathbb{R}$ of course  $\kappa\cdot X=X$, while for $\kappa=-\mathrm{id}_\mathbb{R}$ we get the dual random walk  $\kappa\cdot X=-X$ (whose minima are the maxima of $X$). It is plain that ``generically'' transformations  $\kappa$ that are not  a multiple of $\mathrm{id}_\mathbb{R}$ will yield many symmetric splitting families other than the times of maxima and minima of $X$. 
\end{example}

Combining Corollaries~\ref{corollary:preclude} and~\ref{corollary} we arrive at

\begin{proposition}\label{proposition:nasc-for-existence-of-symmetric-splitting}
The following are equivalent.
\begin{enumerate}[(a)]
\item\label{nasc.a} $(\PP;\FF,\Delta,\rr)$ admits a symmetric splitting family. 
\item\label{nasc.b} $\PP\vert_{\FF_1}$ is atomless and there exists an $A\in \FF_1$ for which $\mathbbm{1}_A+\mathbbm{1}_{A}(\rr\circ \Delta_1)=\mathbbm{1}_\Omega$ a.s.-$\PP$.
 \item\label{nasc.c} There exists a random walk $X$ adapted to $(\PP;\FF,\Delta,\rr)$, whose jumps have diffuse law.
 \end{enumerate}
\end{proposition}
\begin{proof}
\ref{nasc.a} implies \ref{nasc.b} by Corollary~\ref{corollary:preclude} and because, with $\tau$ symmetric splitting, $A:=\{\tau_1=1\}\in \FF_1$ meets the stipulated condition, indeed, by \eqref{main} say, $(\rr\circ \Delta_1)^{-1}(A)=\{\tau_1=0\}=\Omega\backslash A$ a.s.-$\PP$.

\ref{nasc.b} $\Rightarrow$ \ref{nasc.c}. Since $\PP\vert_{\FF_1}$ is atomless, applying \cite[Theorem~10.52(2)]{aliprantis}, there exists a sequence $(A_n)_{n\in \mathbb{N}}$ in $\FF_1\vert_A$ such that $\PP^+:=\PP\vert_\MM$ is an atomless measure on the countably generated $\sigma$-algebra $\MM:=\sigma_A(\{A_n:n\in \mathbb{N} \})$ of  $A$. Then $\xi^+:=\sum_{n\in \mathbb{N}}3^{-n}\mathbbm{1}_{A_n}$, with domain $A$, generates $\MM$, in particular ${\xi^+}_\star \PP^+$ is diffuse. Setting $\xi:=\mathbbm{1}_A\xi^+-\mathbbm{1}_{(\rr\circ \Delta_1)^{-1}A}\xi^+(\rr\circ \Delta_1)$ we obtain a random variable of $\PP$ with diffuse law and satisfying $\xi=-\xi(\rr\circ \Delta_1)$. From Remark~\ref{rmk:adapted-rws} we deduce existence of a random walk adapted to $(\PP;\FF,\Delta,\rr)$, whose jumps have diffuse law.

From Corollary~\ref{corollary} we see that  \ref{nasc.a} is necessary for \ref{nasc.c}.
\end{proof}

\subsection{Honest splitters}\label{honest}
Let us consider the doubly-indexed family $(\tau_{m,n})_{(m,n)\in \mathbb{Z},m\leq n}$ defined via 
\begin{equation}\label{def-doubly-indexed}
\tau_{m,n}:=m+\tau_{n-m}( \Delta_m),\quad (m,n)\in\mathbb{Z}^2,\, m\leq n.
\end{equation}
Manifestly all the symmetric splitting times of Corollary~\ref{corollary}  have the property of
\begin{definition}\label{definition:honest}
The family $\tau$ is said to be honest if $\tau_{m,n}=\tau_{k,l}$ a.s.-$\PP$ on $\{\tau_{m,n}\in [k,l]\}$ for all $m\leq k\leq l\leq n$ from $\mathbb{Z}$.
\end{definition}
It is immediate from Theorem~\ref{thm:discrete} that 
\begin{equation}\label{eq:condition-for-symmetry-given-honest}
\text{an honest $\tau$ is symmetric splitting iff it verifies \eqref{symmetric-split-eq},}
\end{equation}
 while for a symmetric splitting $\tau$  equiveridical are:
\begin{enumerate}[(a)]
\item the family $\tau$ is honest;
\item \label{equiveridical.b} for all $m\leq n$ from $\mathbb{N}_0$, $\PP$-a.s., $\{\tau_n=0\}\subset \{\tau_m=0\}$.
\end{enumerate}
It is not known to the author whether there exist  dishonest, as it were,  symmetric splitting times. But the structure of honest symmetric splitting families is particularly simple, indeed each corresponds to precisely one stopping time of $(\FF_n)_{n\in \mathbb{N}}$ that is self-dual in the sense of

\begin{definition}\label{definition:self-dual}
For a stopping time $\gamma$ in the filtration $(\FF_n)_{n\in \mathbb{N}}$ (with values in $\mathbb{N}\cup \{\infty\}$) introduce
\begin{align*}
\gamma^0&:=0; \text{ then inductively the ``$\gamma$-iterates''}\\
\gamma^{n+1}&:=\gamma^n+\mathbbm{1}_{\{\gamma^n<\infty\}}\gamma(\Delta_{\gamma^n}),\quad n\in \mathbb{N}_0;\\
\mathcal{R}\llbracket\gamma\rrbracket&:=\{\gamma^n:n\in \mathbb{N}_0\}\cap [0,\infty) ;\\
 \llbracket\gamma\rrbracket_n&:=\max ([0,n]\cap \mathcal{R}\llbracket\gamma\rrbracket),\quad n\in \mathbb{N}_0;\quad \llbracket\gamma\rrbracket:=(\llbracket\gamma\rrbracket_n)_{n\in \mathbb{N}_0}.
\end{align*}
Call such a $\gamma$ self-dual for $(\PP;\FF,\Delta,\rr)$  if for all $n\in \mathbb{N}_0$, $\PP$-a.s.,
\begin{equation}\label{self-dual}
\{n\in \mathcal{R}\llbracket\gamma\rrbracket\}=(\rr\circ \Delta_n)^{-1}\{n<\gamma\}.
 \end{equation}
\end{definition}
Compare with \cite[Definition~1]{KENNEDY}: unraveling the notation, it is seen quickly that, within the confines of Example~\ref{ex.increments} (which  is not a severe restriction), self-duality of a $\gamma$ is exactly (modulo only the a.s. qualifiers in our case and trivial transpositions) the duality of $\gamma$ with $\hat\gamma:=\gamma\circ \hat\iota$ in the sense of  \cite[Definition~1]{KENNEDY}; here $\hat\iota:=( E^\mathbb{Z}\ni \theta\mapsto (\ZZ\ni k\mapsto \iota(\theta(k))))$ is the component-wise action of the involutive map $\iota$. 

It is also elementary to check that, in the context of Definition~\ref{definition:self-dual},
\begin{enumerate}[(A)]
\item \label{A} self-duality \eqref{self-dual}  of $\gamma$ is equivalent to $\{\llbracket\gamma\rrbracket_n=0\}=(\rr\circ \Delta_n)^{-1}\{\llbracket\gamma\rrbracket_n=n\}$ holding true a.s.-$\PP$ for all $n\in \mathbb{N}_0$;
\item\label{B} the set $\mathcal{R}\llbracket\gamma\rrbracket$ is regenerative in the sense that $\mathcal{R}\llbracket\gamma\rrbracket\cap [n,\infty)=n+\mathcal{R}\llbracket\gamma\rrbracket(\Delta_n)$ on $\{n\in \mathcal{R}\llbracket\gamma\rrbracket\}$ for all $n\in \mathbb{N}_0$;
\item\label{C} $\mathcal{R}\llbracket\gamma\rrbracket=\{n\in \mathbb{N}_0:\llbracket\gamma\rrbracket_n=n\} $ and $\gamma=\inf  (\mathcal{R}\llbracket\gamma\rrbracket \cap (0,\infty))$.
\end{enumerate}

As for the announced characterization of honest symmetric splitters via self-dual stopping times, it is the content of
\begin{proposition}\label{proposition:self-dual}
There is up to $\PP$-a.s. equality a one-to-one and onto correspondence between honest symmetric splitting families $\tau$ and self-dual stopping times $\gamma$ for $(\PP;\FF,\Delta,\rr)$ that is given by the mutually inverse associations
\begin{equation*}
\begin{cases}
\gamma\rightsquigarrow \llbracket\gamma\rrbracket\text{ and} \\
\tau\rightsquigarrow \llparenthesis \tau\rrparenthesis:=\inf \{n\in \mathbb{N}:\tau_n=n\}
\end{cases}.
\end{equation*}
\end{proposition}
 In particular, specializing to 
Example~\ref{example:basic}, the fact that $\tau^0$ is symmetric splitting for $X$ corresponds to the first descending and ascending ladder times, respectively $$\gamma^0:=\min \{n\in \mathbb{N}:X_n<0\}\text{ and }\widehat{\gamma^0}:=\min \{n\in \mathbb{N}:X_n>0\},$$ being in duality, i.e. to $\gamma^0$ being self-dual.
\begin{proof}
First we argue that for a symmetric splitting $\tau$, $\llparenthesis \tau\rrparenthesis$ is self-dual and $\llbracket\llparenthesis \tau\rrparenthesis\rrbracket= \tau$ a.s.-$\PP$. The honesty property of $\tau$ implies in fact that $\PP$-a.s. $\tau=(\mathbb{N}_0\ni n\mapsto \tau_n)$ is a map that lies below the diagonal of $\mathbb{N}_0\times \mathbb{N}_0$, vanishes at zero and is constant to the right except for jumps to the diagonal corresponding to the regenerative (in the sense of \ref{B}) set of times $\{n\in \mathbb{N}_0:\tau_n=n\}$ that coincides with  $\mathcal{R}\llbracket\llparenthesis \tau\rrparenthesis\rrbracket$ a.s.-$\PP$. It is then evident that for each $n\in \mathbb{N}_0$, $\PP$-a.s., $$\tau_n=\max ([0,n]\cap \mathcal{R}\llbracket\llparenthesis \tau\rrparenthesis\rrbracket),$$ i.e. $\llbracket\llparenthesis \tau\rrparenthesis\rrbracket=\tau$ a.s.-$\PP$. Besides, the self-duality property of $\llparenthesis \tau\rrparenthesis$ becomes, in view of \ref{A}, just the requirement that $\{\tau_n=0\}=\{\tau_n(\rr \circ \Delta_n)=n\}$ a.s.-$\PP$ for all $n\in \mathbb{N}_0$, which is a special case of \eqref{symmetric-split-eq}. 

Conversely, let us establish that for a self-dual $\gamma$, $\llbracket\gamma\rrbracket$ is an honest symmetric splitting family (from \ref{C} it already follows that $\llparenthesis\llbracket\gamma\rrbracket\rrparenthesis=\gamma$). Let then $\{m,n\}\subset \mathbb{N}_0$. On $\{\llbracket\gamma\rrbracket_{m+n}=m\}$ we have ($m\in \mathcal{R}\llbracket\gamma\rrbracket$ and hence) $\llbracket\gamma\rrbracket_m=m$, thus $\PP$-a.s. $\llbracket\gamma\rrbracket_m(\rr \circ \Delta_m)=0$ due to the self-duality \ref{A}, while by the regenerative property \ref{B} it is also clear that $\llbracket\gamma\rrbracket_n(\Delta_m)=0$.  This gives the ``$\subset$'' inclusion of \eqref{main} (for $\tau=\llbracket\gamma\rrbracket$). On the other hand, on $\{\llbracket\gamma\rrbracket_m(\rr \circ \Delta_m)=0\}$ we must have $\llbracket\gamma\rrbracket_{m+n}\geq m$ a.s.-$\PP$: if, per absurdum, for some $k\in \{0,\ldots,m-1\}$, $\llbracket\gamma\rrbracket_{m+n}=k$, then, $\PP$-a.s., by the inclusion just argued we have   $\llbracket\gamma\rrbracket_{m+n-k}(\Delta_k)=0$, a fortiori $\llbracket\gamma\rrbracket_{m-k}(\Delta_k)=0$, hence by self-duality \ref{A} $\llbracket\gamma\rrbracket_{m-k}(\rr\circ  \Delta_m)=m-k>0$, therefore $\llbracket\gamma\rrbracket_m(\rr\circ  \Delta_m)>0$, which cannot stand. Similarly, on $\{\llbracket\gamma\rrbracket_n(\Delta_m)=0\}$ we must have $\llbracket\gamma\rrbracket_{m+n}\leq m$ a.s.-$\PP$: if, per absurdum, for some $k\in \{m+1,\ldots,m+n\}$, $\llbracket\gamma\rrbracket_{m+n}=k$, then, $\PP$-a.s.,  $\llbracket\gamma\rrbracket_{k}(\rr \circ \Delta_k)=0$, a fortiori $\llbracket\gamma\rrbracket_{k-m}(\rr \circ \Delta_k)=0$, so by self-duality \ref{A} $\llbracket\gamma\rrbracket_{k-m}(\Delta_{m})=k-m>0$, therefore $\llbracket\gamma\rrbracket_n(\Delta_m)>0$, which again  cannot be.  Altogether we have then also the ``$\supset$'' inclusion of \eqref{main} (for $\tau=\llbracket\gamma\rrbracket$). By Theorem~\ref{thm:discrete} we deduce that $\llbracket\gamma\rrbracket$ is symmetric splitting. It is also honest, since \ref{equiveridical.b} of p.~\pageref{equiveridical.b} holds evidently.
\end{proof}
Thanks to Proposition~\ref{proposition:self-dual} and the comments just preceding it, the results of \cite{greenwood-shaked,KENNEDY} for sure apply to all honest symmetric splitting times of symmetric real-valued random walks (recall Examples~\ref{example:basic}-\ref{ex.increments}). Especially, one has for them a generalized Spitzer-Pollaczek factorization \cite[Eq. (7)]{greenwood-shaked} \cite[Eqs.~(4)-(5)]{KENNEDY}.

In closing this (sub)section we provide a constructive exhaustion\footnote{Due to Jon Warren (private communication).} of honest symmetric splitting times, which may be compared with \cite[Corollary~2]{KENNEDY} for the case of stopping times that are in duality. 

  First, set $\tau^J_{0}:=0$. Inductively, given an $n\in \mathbb{N}_0$, suppose the  $\tau^J_m$ have been defined for all $m\in [0,n]\cap  \ZZ$ and that for all $k\leq  l$ from $[0,n]\cap \ZZ$, $\tau^J_{n}=\tau^J_{k,l}:=k+\tau^J_{k-l}(\Delta_l)$ a.s.-$\PP$ on $\{\tau^J_{n}\in [k,l]\}$, but also, for all $m\in  [0,n]\cap  \ZZ$, $\tau^J_m+\tau^J_m(\rr\circ \Delta_m)=m$ a.s.-$\PP$ (recall \eqref{eq:condition-for-symmetry-given-honest}). Then set (there may be ambiguity on a $\PP$-negligible set, which does not matter)
$$\tau^J_{n+1}:=
\begin{cases}
\tau^J_{1,n+1}&\text{ on }\{\tau^J_n\in [n]\}\\
\tau^J_{n}&\text{ on }\{\tau^J_{1,n+1}\in [n]\}\\
n+1&\text{ on }J_{n+1}\\
0&\text{ on }I_{n+1}\backslash J_{n+1}\\
\end{cases}
$$
for
$J_{n+1}\subset I_{n+1}:=\{ \tau^J_n=0\}\cap \{\tau^J_{1,n+1}=n+1\}$ an arbitrary element of $\FF_{n+1}$ satisfying the property that $J_{n+1}$ and $(\rr \circ \Delta_{n+1})^{-1}J_{n+1}$ are disjoint and their union is $I_{n+1}$ a.s.-$\PP$. It is clear that such a construction precisely exhausts in the $\tau^J$ all the honest symmetric splitting families. 
Notice that, at the ``$n$-th'' inductive step, $n\in \mathbb{N}_0$, the set $I_{n+1}$ satisfies  $(\rr \circ \Delta_{n+1})^{-1}I_{n+1}=I_{n+1}$ a.s.-$\PP$, so that existence of some $J_{n+1}$ is never an issue as long as there exists at least one $A_{n+1}\in \FF_{n+1}$ for which $\Omega$ is $\PP$-a.s. the disjoint union of $A_{n+1}$ and of $(\rr \circ \Delta_{n+1})^{-1}A_{n+1}$ (for we may then  take $J_{n+1}:=I_{n+1}\cap A_{n+1}$ irrespective of the choices hitherto made). The existence of such $A_m$ \emph{for all} $m\in \mathbb{N}$ is equivalent to the existence of a symmetric splitting family, which follows from Proposition~\ref{proposition:nasc-for-existence-of-symmetric-splitting} and the observation that  if $X$ is adapted to $(\PP;\FF,\Delta,\rr)$ and its jumps have diffuse law, then we may plainly take $A_{n+1}=\{X_1+\cdots +X_{n+1}<0\}$ for all $n\in \mathbb{N}_0$ (in fact, putting in such case  $J_{n+1}=I_{n+1}\cap A_{n+1}$ at each step $n\in \mathbb{N}_0$ of the induction simply recovers in $\tau^J$ the minima $\tau^0$ of $X$ from Corollary~\ref{corollary}).

\section{Remarks on Brownian motion}\label{section:bm}
Let $\Omega$ be the space of continuous real maps on $\mathbb{R}$ vanishing at zero and denote the canonical process on $\Omega$ by $B$. Endow $\Omega$ with the $\sigma$-algebra $\GG$ generated by $B$ and the two-sided Wiener measure $\mathbb{W}$ thereon. For real $s\leq t$ denote by $\FF_{s,t}$ the $\sigma$-algebra generated by the increments of $B$ on $[s,t]$ and the $\mathbb{W}$-trivial sets. The time-shifts $\Delta=(\Delta_h)_{h\in \mathbb{R}}$ and the reflection of time $\rr$ are defined in the natural way: $\Delta_h:=(\Omega\ni \omega\mapsto \omega(h+\cdot)-\omega(h))$ for $h\in \mathbb{R}$ and $\rr:=(\Omega\ni \omega\mapsto \omega(-\cdot))$. It is plain that the structure $(\WW;\FF,\Delta,\rr)$ satisfies the obvious analogs of the properties \eqref{properties:i}-\eqref{properties:v} from the discrete setting. In particular, $(\WW;\FF,\Delta)$ constitutes a one-dimensional noise in the sense of Tsirelson, the classical Wiener noise.

Then, the problem is to explore the space of measurable processes $\tau=(\tau_t)_{t\in (0,\infty)}$ that are adapted: $\tau_t\in (0,t)$ a.s.-$\WW$ and $\tau_t$ is $\FF_{0,t}$-measurable for each $t\in (0,\infty)$; and which have the symmetric splitting property: for each $\lambda\in (0,\infty)$, with $\ee$ an exponential random time of rate $\lambda$ independent of $\GG$, relative to the $\sigma$-field $\FF_\infty^0\otimes \BB_{(0,\infty)}$, the pair $(\rr\circ \Delta_{\tau_\ee},\tau_\ee)$  is independent of and has the same distribution as the pair $(\Delta_{\tau_\ee},\ee-\tau_\ee)$. Here $\FF_\infty^0:=\sigma(B\vert_{[0,\infty)})$ (no completion!). We see that it is exactly parallel to the discrete setting, except that we cannot work with $\FF_\infty:=\lor_{t\in [0,\infty)}\FF_{0,t}$ directly; the notion of the ``future information'' $\FF^0_\infty$ is a separate extra ingredient (albeit one satisfying $ \FF_\infty=\FF^0_\infty\lor\PP^{-1}(\{0,1\})$). The reason for this is as follows: if $\tau^0$ is the family of the minima of $B$, then  for $a\in (0,\infty)$, the event $E_a:=\{B\text{ is strictly positive on $[-a,a]\backslash \{0\}$}\}$ has $\WW$-probability zero (so belongs to $\FF_\infty$), but $\{\rr\circ \Delta_{\tau^0_\ee}\in E_a\}=\{\Delta_{\tau^0_\ee}\in E_a\}$ is non-trivial (so not independent of itself). Note this issue does not arise in the discrete setting of Section~\ref{discrete} at all, simply because there ${\Delta_\rho}_\star\PP\ll \PP$ for all $\ZZ$-valued random times $\rho$. Of course the reader may complain that we should not have completed the $\FF_{s,t}$ to begin with, however the latter is more natural from the point of view of the theory of noises. 

In any event, with $\FF_\infty^0$ in lieu of $\FF_\infty$, the notion appears to be well-posed. In view of the results of the discrete setting it is particularly interesting, because there are no jumps of Brownian motion on which an odd transformation could be effected. There are moreover no symmetric L\'evy processes adapted to the Wiener noise other than multiples of the underlying Brownian motion itself  \cite[p.~184]{tsirelson-nonclassical}. It begs therefore the question: are there any adapted symmetric splitting families for $B$ at all, apart from  the maxima and the minima? If there are, are there any that have the added property of honesty/nestedness (cf. \eqref{def-doubly-indexed} and Definition~\ref{definition:honest}): $\tau_{s,t}:=s+\tau_{t-s}(\Delta_{s})=\tau_{u,v}$ a.s.-$\WW$ on $\{\tau_{s,t}\in (u,v)\}$ for all real $s\leq u< v\leq t$?

To indicate some of the subtleties involved, consider only the narrower problem of nested families. The resulting doubly-indexed structure $(\tau_{s,t})_{s<t}$ has been previously studied in \cite[Section~5]{vidmar-warren},  were it was called an honest indexation \cite[Definition~5.1]{vidmar-warren} (the emphasis there being on the range $M:=\{\tau_{s,t}:(s,t)\in \mathbb{Q}^2,\, s<t\}$, which does not matter here). It was found  \cite[p.~1101]{vidmar-warren} that if $\tau$ is also symmetric, in that $\tau_t+\tau_t( \rr\circ \Delta_t)=t$ a.s.-$\WW$ for all $t\in (0,\infty)$, then $\tau$ is symmetric splitting (cf. \eqref{eq:condition-for-symmetry-given-honest}) and the question was raised \cite[Question~5.8]{vidmar-warren} whether in addition to the minima  $\tau^0$ and the  maxima $\tau^0(-B)$ there are any other symmetric honest indexations. But even in this (a priori) more restricted setting, attempting to resolve the question appears quite non-trivial. To give (further) credence to this, let us conclude by reporting, slightly informally, and without including all the detail, on one failed strategy that was aimed at showing that the answer is to the negative. 

Just like in the discrete case, it is concluded immediately that for  any symmetric honest indexation the law of $\tau_t$ is uniquely determined for each $t\in (0,\infty)$. Exploiting the fact that a local time $L$ may be associated to the set $D:=\{t\in (0,\infty):\tau_t=t\}=\{\tau_t:t\in (0,\infty)\}$ (after a suitable regularization of $\tau$)  \cite[Proposition~5.12]{vidmar-warren}, the range of $\tau$ is seen to be a.s. the closure of the range of the subordinator $L^{-1}$ (:= the right-continuous inverse of $L$). This means that the law of $L^{-1}$ is uniquely determined (up to a temporal scaling) \cite[Lemma~1.11]{Bertoin1999} and hence in turn so is that of $(\tau_t)_{t\in (0,\infty)}$ as   a process. If it could be shown that moreover the law of the entire doubly-indexed family $(\tau_{s,t})$ is uniquely determined, then we could construct an endomorphism of $\WW$ into itself by sending $(\tau^0_{s,t})$ onto $(\tau_{s,t})$ and this endomorphism would preserve the noise structure of $(\WW;\FF,\Delta)$. We would have an endomorphism of the Wiener noise. Since the latter admits no non-trivial subonoises, the endomorphism in question would in fact be an automorphism and could thus only be the identity or minus the identity on $\Omega$, concluding the argument. 

But, in ``\emph{If it could be shown /\ldots/}'' lies the rub. Indeed, adding to the Brownian motion $B$ an independent non-zero  symmetric compound Poisson process, call the resulting L\'evy process $X$, we may take for $\tau^1$ the minima of $X$, while $\tau^0$  continue to be the minima of $B$. After making the suitable notions precise it is morally clear (it would ultimately have to be proved, of course) that $\tau^1$ and $\tau^0$ are both symmetric splitting for $X$, and $(\tau^1_t)_{t\in (0,\infty)}$ has the same law as $(\tau^0_t)_{t\in (0,\infty)}$, however $(\tau^1_{s,t})$ as a doubly-indexed process does not have the same law as $(\tau^0_{s,t})$ (just because, in the obvious notation, for $s\in (0,\infty)$, the coupling time $\inf\{t\in (s,\infty):\tau^0_t=\tau^0_{s,t}\}-s$ is a.s. the first time that the local time $L^0(\Delta_s)$ accrues to $L^0_s(\rr\circ \Delta_s)$, while this fails for $\tau^1$). Thus, if in fact in the case of $B$ only the maxima and minima are symmetric splitting, then it must be some property particular to Brownian motion (over and above it being a L\'evy process) that necessitates that the law of the doubly-indexed process $(\tau_{s,t})$ is uniquely determined by the requirement that the family be symmetric splitting. It is altogether not clear to the author what this property might be.

\bibliography{Biblio_noise}

\begin{thebibliography}{10}

\bibitem{aliprantis}
C.~D. Aliprantis and K.~C. Border.
\newblock {\em Infinite dimensional analysis: a hitchhiker's guide}.
\newblock Springer-Verlag Berlin Heidelberg, third edition, 2006.

\bibitem{bertoin}
J.~Bertoin.
\newblock {\em {L\'e}vy Processes}.
\newblock Cambridge Tracts in Mathematics. Cambridge University Press,
  Cambridge, 1996.

\bibitem{Bertoin1999}
J.~Bertoin.
\newblock Subordinators: Examples and applications.
\newblock In P.~Bernard, editor, {\em Lectures on Probability Theory and
  Statistics: Ecole d'Et{\'e} de Probailit{\'e}s de Saint-Flour XXVII - 1997},
  pages 1--91. Springer Berlin Heidelberg, 1999.

\bibitem{splitting-non-standard}
N.~J. Cutland and W.~S. Kendall.
\newblock A non-standard proof of one of {D}avid {W}illiams' splitting-time
  theorems.
\newblock {\em Advances in Applied Probability}, 18:37--47, 1986.

\bibitem{greenwood-pitman}
P.~Greenwood and J.~Pitman.
\newblock Fluctuation identities for {L\'e}vy processes and splitting at the
  maximum.
\newblock {\em Advances in Applied Probability}, 12(4):893--902, 1980.

\bibitem{greenwood-shaked}
P.~Greenwood and M.~Shaked.
\newblock Dual pairs of stopping times for random walk.
\newblock {\em The Annals of Probability}, 6(4):644--650, 1978.

\bibitem{kallenberg-splitting}
O.~Kallenberg.
\newblock Splitting at backward times in regenerative sets.
\newblock {\em The Annals of Probability}, 9(5):781--799, 1981.

\bibitem{KENNEDY}
J.~E. Kennedy.
\newblock On duality and the {S}pitzer-{P}ollaczek factorization for random
  walks.
\newblock {\em Stochastic Processes and their Applications}, 76(2):251--266,
  1998.

\bibitem{kyprianou}
A.~E. Kyprianou.
\newblock {\em Fluctuations of {L}\'evy processes with Applications:
  Introductory Lectures}.
\newblock Universitext. Springer Heidelberg, second edition, 2014.

\bibitem{perman}
M.~Perman.
\newblock A decomposition for {M}arkov processes at an independent exponential
  time.
\newblock {\em Ars Mathematica Contemporanea}, 12(1):51--65, 2017.

\bibitem{rogers-williams}
L.~C.~G. Rogers and D.~Williams.
\newblock {\em Diffusions, {M}arkov Processes and Martingales: {V}olume 1,
  {F}oundations}.
\newblock Cambridge Mathematical Library. Cambridge University Press, 2000.

\bibitem{tsirelson-nonclassical}
B.~Tsirelson.
\newblock Nonclassical stochastic flows and continuous products.
\newblock {\em Probability Surveys}, 1:173--298, 2004.

\bibitem{picard2004lectures}
B.~Tsirelson.
\newblock Scaling limit, noise, stability.
\newblock In J.~Picard, editor, {\em Lectures on Probability Theory and
  Statistics: Ecole {d'Et\'e} de Probabilit{\'e}s de Saint-Flour XXXII - 2002},
  Lecture Notes in Mathematics, pages 1--106. Springer Berlin Heidelberg, 2004.

\bibitem{independence-vidmar}
M.~Vidmar.
\newblock Independence times for iid sequences, random walks and {L}évy
  processes.
\newblock {\em Stochastic Processes and their Applications},
  129(10):3619--3637, 2019.

\bibitem{vidmar-warren}
M.~Vidmar and J.~Warren.
\newblock Stationary local random countable sets over the {W}iener noise.
\newblock {\em Probability Theory and Related Fields}, 188(3):1063--1129, 2024.

\end{thebibliography}
\bibliographystyle{plain}

\end{document}